\documentclass[12pt]{article}
	\usepackage{amssymb,amsfonts,amsmath,amsthm, breqn, color, float, mathtools, booktabs, hvfloat, url, enumerate,tikz}
	\usepackage{graphicx,color}
	\usepackage{blindtext}

	\usepackage{mathrsfs}

	\usepackage{adjustbox}
	\usepackage{caption}
	\def\obrace{\iftrue{\else}\fi}
	\def\cbrace{\iffalse{\else}\fi}
	\let\originalparagraph\paragraph

	\renewcommand{\paragraph}[2][.]{\originalparagraph{#2#1}}

	\newcommand{\rr}{{\mathbb R}}
	\newcommand{\cc}{{\mathbb C}}

	\newcommand{\cala}{{\mathcal A}}
	
	\newcommand{\cald}{{\mathcal D}}

	\newcommand{\calj}{{\mathcal J}}

	\newcommand{\cals}{{\mathcal S}}

	\newcommand{\calt}{{\mathcal T}}

	\newcommand{\beq}{\begin{eqnarray*}}
		\newcommand{\feq}{\end{eqnarray*}}
	\newcommand{\beqn}{\begin{eqnarray}}
	\newcommand{\feqn}{\end{eqnarray}}

	\newtheorem{theorem}{Theorem}
	\newtheorem*{conj*}{Conjecture}
	\makeatletter \@addtoreset{theorem}{section}\makeatother

	\makeatletter \@addtoreset{theorem}{section}\makeatother

	\makeatletter \@addtoreset{theorem}{section}\makeatother
	
	\newtheorem{lemma}[theorem]{Lemma}
	
	\newtheorem*{theorem*}{Theorem}
	
	\newtheorem{proposition}[theorem]{Proposition}

	\newtheorem{example}[theorem]{Example}
	\usepackage{algorithm}
	\usepackage{graphicx}
	\usepackage[noend]{algpseudocode}
	\makeatletter
	\def\BState{\State\hskip-\ALG@thistlm}
	\makeatother

	\DeclareCaptionLabelFormat{noname}{#2}
	\newlength\myindent
	\DeclareMathOperator{\E}{\mathbb{E}}
	\DeclareMathOperator{\Var}{\mathbb{V}ar}
	\DeclareMathOperator{\PP}{\mathbb{P}}

	\title{On typical triangulations of a convex $n$-gon}
	\author{Toufik~Mansour\thanks{ Department of Mathematics, University of Haifa, 199 Abba Khoushy Ave, 3498838 Haifa, Israel;
			\newline e-mail: tmansour@univ.haifa.ac.il}
		\and
		Reza~Rastegar\thanks{Occidental Petroleum Corporation, Houston, TX 77046 and Departments of Mathematics and Engineering, University of Tulsa, OK 74104, USA - Adjunct Professor; e-mail:  reza\_rastegar2@oxy.com}
	}
	
\begin{document}
		\maketitle
		\begin{abstract}
		Let $f_n$ be a function assigning weight to each possible triangle whose vertices are chosen from vertices of a convex polygon $P_n$ of $n$ sides. Suppose $\calt_n$ is a random triangulation, sampled uniformly out of all possible triangulations of $P_n$. We study the sum of weights of triangles in $\calt_n$ and give a general formula for average and variance of this random variable. In addition, we look at several interesting special cases of $f_n$ in which we obtain explicit forms of generating functions for the sum of the weights. For example, among other things, we give new proofs for already known results such as the degree of a fixed vertex and the number of ears in $\calt_n,$ as well as, provide new results on the number of ``blue" angles and  refined information on the distribution of angles at a fixed vertex. We note  that our approach is systematic and can be applied to many other new examples while generalizing the existing results.
		\end{abstract}
		
		
		\section{Introduction}
		\label{intro}
		
		We consider a convex polygon $P_n$ with $n$ vertices and label the vertices $V_n:=\{v_{n,j}\}_{\{1 \leq j \leq  n\}}$ in clockwise order. A triangulation is a set of $n-3$ noncrossing diagonals $v_{n,i}v_{n,j}$ with $1\leq i\neq j \leq n$
		which partitions $P_n$ into $n-2$ triangles. Triangulation is a classical area of research going back to at least Euler. He showed the number of possible triangulations for $P_n$ is $C_{n-2}$ where $C_n=\frac{1}{n+1}\binom{2n}{n}$ is the $n$-th Catalan number. Triangulation has been extended to general point sets residing in various spaces and manifolds and also found  many applications in computer science, computer graphics, and mathematics. We refer to \cite{goodman1, jesus1} and references within for a comprehensive review. The theme of this paper is with respect to the properties of a typical triangulation $\calt_n$ of $P_n$. Studying $\calt_n$ was initiated in a paper of Poly\'a \cite{polya1} published in American Math Monthly in 1956. Among of large literature published on the subject, we refer to \cite{nicla1, flajolet1, gao1, noy2, noy3, regev2} where, among other things, several aspects of $\calt_n$ including the maximum degree of vertices, the longest diagonal, the number of ears, the number of triangles with a side parallel to a fix side of $P_n$ are studied. Our objective in this paper is to develop a somewhat systematic approach to address similar questions on $\calt_n$. To that end, we first formalize the property of interest by defining a function that assigns weights to the triangles of each triangulation. Through a simple constructive algorithm that samples a uniform triangulation of $P_n$, we next derive a system of recursive equations for the generating functions corresponding to that function. We then leverage certain invariance properties of the function of interest to reduce the generating functions to solvable forms. By obtaining explicit information on these generating functions, we are finally able to describe the random triangulation with respect to the property of interest. To elaborate our approach, we give new proofs for already known results, and in addition, discuss a few new examples.
		
		We start with stating a few notations. Throughout this paper, $\rr$ and $\cc$ refer to the set of all real and complex numbers. Let $P_{n,l,r}$ be the convex-hull of vertices $V_{n,l,r}:=\{v_{n,j}\}_{\{l \leq j \leq  r\}}.$ With this notation, $P_n:=P_{n, 1, n}$ is the polygon of interest with $n$ vertices and $P_{n,l,r}$ is a convex polygon with $m:=r-l+1$ sides. Let $$T_{n,l,r}:=\{T_{n,l,r,1}, \cdots, T_{n,l,r,C_{m-2}}\}$$  be the set of all triangulations of $P_{n,l,r}.$ Suppose that we choose a triangulation $\calt_n$ out of $C_{n-2}$ triangulations in the set $T_{n}$ (set $T_n:=T_{n,1,n}$) with uniform probability $\PP:T_{n}\to[0,1]$. In the following, we use $\E$ and $\Var$ to refer to the expectation and the variance with respect to $\PP$. Let $\Gamma_n$ be the set of all triangles whose vertices are in $V_n.$
		Define $f_n:\Gamma_n \to \cc$ to be a function assigning weights to triangles in $\Gamma_n$. See Figure~\ref{fig:fig2} for an example of $\calt_n$ and how $f_n$ assigns weights. \\
		
		\begin{figure}[h]
			\centering
			\begin{tabular}{c@{}c@{\hspace{1cm}}c@{}}
				\includegraphics[page=1,width=0.45\textwidth]{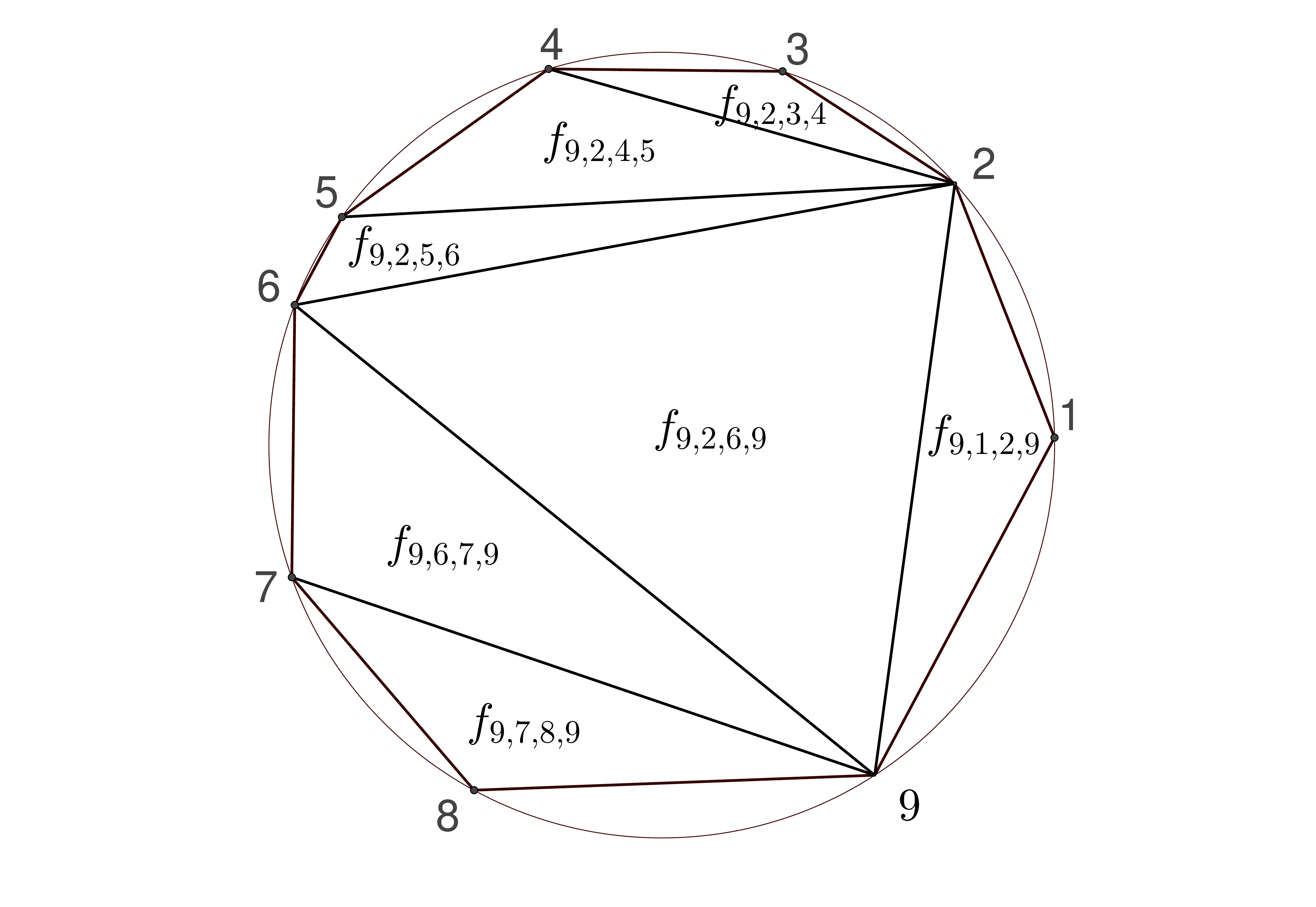} &
			\end{tabular}
			\caption{A random triangulation of an irregular $P_9$. The function $f_9$ assigns weights to all the triangles in $\calt_9$.}
			\label{fig:fig2}
		\end{figure}
		
		Let $\calt_{n,l,r}$ be a random triangulation drawn from $T_{n,l,r}$ with probability $C_{r-l-1}^{-1}$ and
		\beq
		\cals_{n,l,r} = \sum_{\Delta \in \calt_{n,r,l}} f_n(\Delta),
		\feq
		to be the sum of weights of triangles in $\calt_{n,l,r}$. We define the generating function of $\cals_{n,l,r}$ as
		\beq \label{gnrl-def}
		g_{n,l,r}(z) =  \E(z^{\cals_{n,l,r}}) \quad \text{for} \quad z\in \cc.
		\feq
		Clearly, $\calt_n$ is $\calt_{n,1,n}.$ In the following, we set $\cals_n :=\cals_{n,1,n},$ $g_n(z):=g_{n,1,n}(z),$ and $f_{n,l,j,r}:=f_n(\Delta_{n,l,j,r})$ where we use $\Delta_{n,l,j,r}$ to refer to the triangle with three vertices $v_{n,l}, v_{n,j}, v_{n,r}\in V_n$. In our presentation, we always sort the indexes such that $l<j<r$.

		Our first result gives the expectation $E(\cals_n)$ and variance $\Var(\cals_n)$ for a large class of functions $f_n$.
		
		\begin{theorem}\label{coh1}
			Suppose $f_n$ is a function where $f_{n,l,j,r}$ depends only on $r-j,$ $r-l,$ $j-l$ and possibly $n$. For all $n\geq2$,
			\begin{enumerate}
				\item $\E(\cals_{n,l,n})=\frac{1}{C_{n-l-1}}\sum_{j=l}^{n-2}\beta_{n,j}\binom{2j-2l}{j-l},$
				where		
				\beqn
				\beta_{n,j}=\sum_{s=j+1}^{n-1} f_{n,j,s,n}C_{s-j-1}C_{n-j-1}. \label{beta_nl}
				\feqn
				When $l=1,$ \eqref{beta_nl} gives us $E(\cals_n).$
				\item $\Var(\cals_n) = \frac{1}{C_{n-2}} \sum_{j=1}^{n-2}\lambda_{n,j}\binom{2j-2}{j-1} - \frac{1}{C_{n-2}} \left(  \sum_{j=1}^{n-2}\beta_{n,j}\binom{2j-2}{j-1} \right)^2 ,$
				where $\beta_{n,j}$ is given by \eqref{beta_nl} and
				\begin{align}
				\lambda_{n,s}=\sum_{j=s+1}^{n-1}C_{j-s-1} C_{n-j-1}\biggl(&f_{n,s,j,n}^2+2f_{n,s,j,n}(E(\cals_{n,s+n-j,n}) \\ &+E(\cals_{n,j,n}))+ 2E(\cals_{n,s+n-j,n})E(\cals_{n,j,n})\biggr). \notag
				\end{align}			
			\end{enumerate}
		\end{theorem}
		
		This general result can be applied to various interesting geometrical examples including the cases where $f_n$ is the perimeter, the area, or the radius of the inscribed circle of the input triangle. In the first case, $\cals_n$ is related to the minimum-weight triangulation problem also known as optimal triangulation in computational geometry. Optimal triangulation is the problem of finding a triangulation of minimal total edge length where an input polygon must be subdivided into triangles that meet edge-to-edge and vertex-to-vertex, in such a way as to minimize the sum of the perimeters of the triangles \cite{jesus1,xu1}. The two later cases are related to Japanese theorem \cite{johnson1} (See Chapter 4, p.193), which indicates that if $f_n$ is radius of inscribed circle of the input triangle, then $\cals_{n}$ is constant. In addition, when $n$ grows to infinity this sum approaches the diameter of circumscribed circle of the circular polygon $P_n$.

		We remark that it is easy to show that $\cals_n$ is a constant if and only if for all quadrilateral components $v_{n,l}v_{n,j}v_{n,i}v_{n,r}$ with $1\leq l < j< i < r \leq n$ we have $f_{n,l,j,i} + f_{n,l,i,r} =f_{n,l,j,r} + f_{n,j,i,r}.$ This follows by a repeated application of the rule, which ``flips'' one diagonal, will generate all the possible triangulations from any given triangulation, with each ``flip" preserving the sum. See Figure \ref{fig:fig1}, where the triangulation (Left) is flipped to (Right) by flipping $v_{9,2}v_{9,9}$ to $v_{9,1}v_{9,6}.$
		
				\begin{figure}[h]
					\centering
					\begin{tabular}{c@{}c@{}}
						\includegraphics[page=1,width=0.45\textwidth]{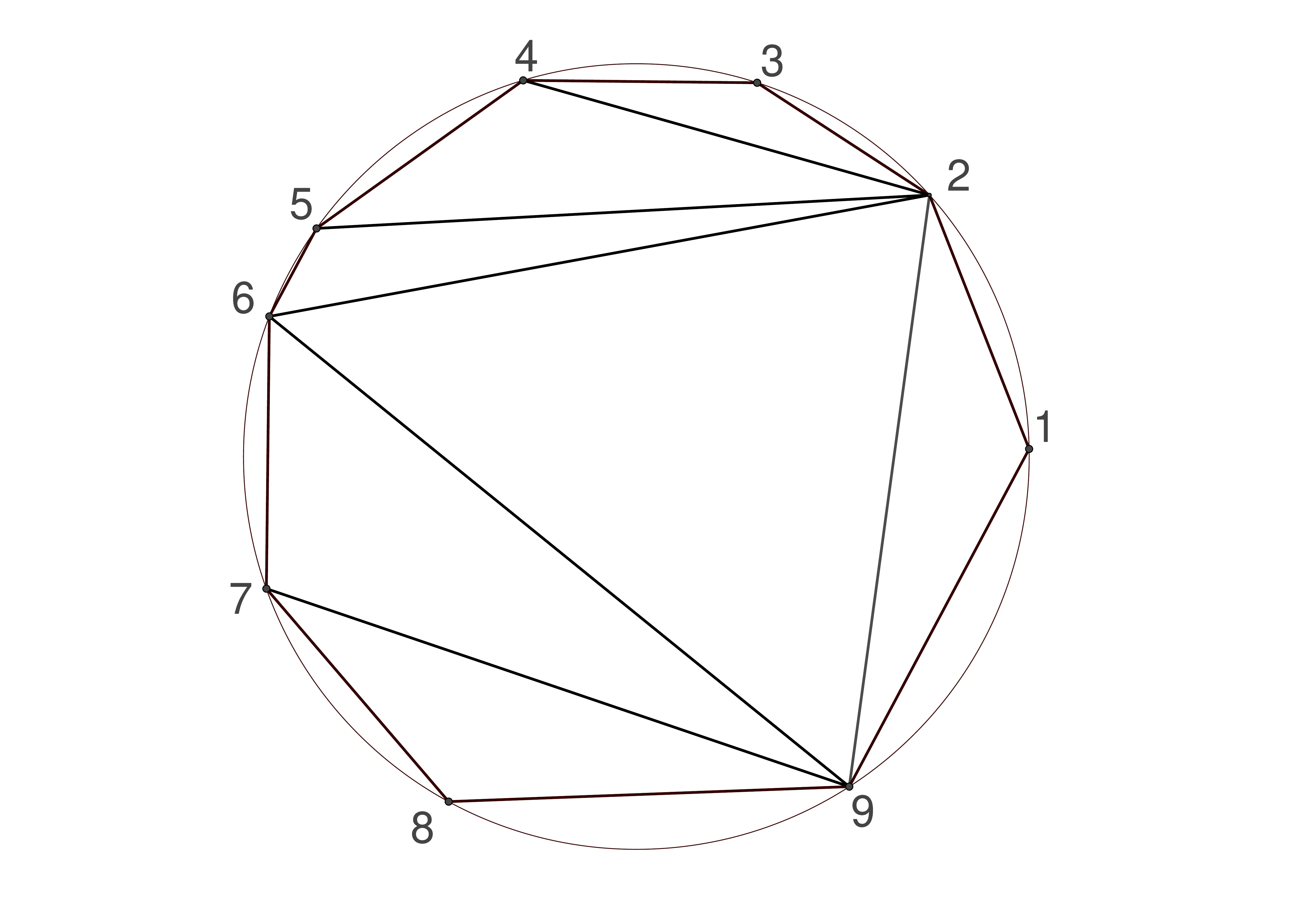} &
						\includegraphics[page=1,width=0.45\textwidth]{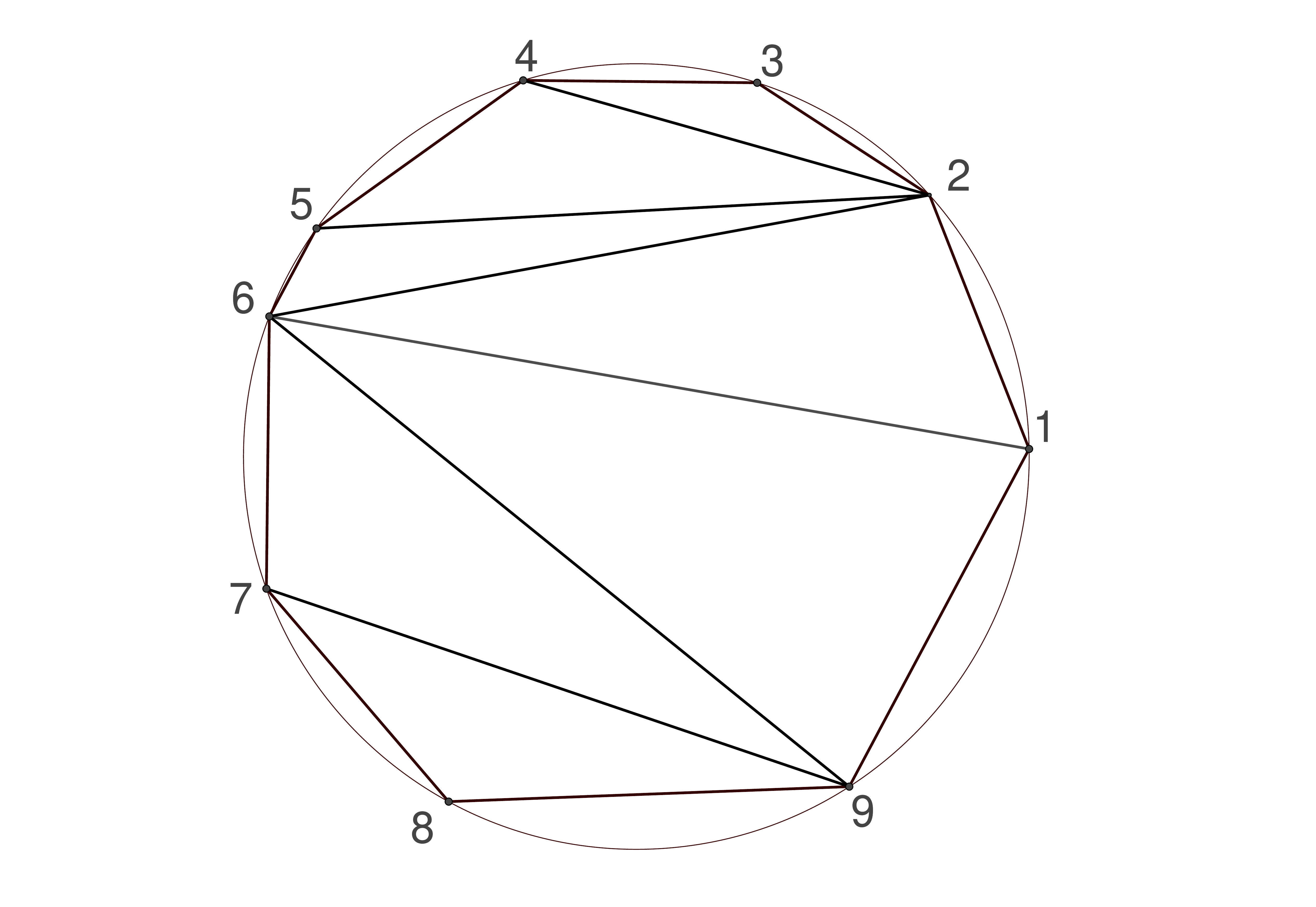} \\						
					\end{tabular}
					\caption{(Left) can be flipped to the (Right) by flipping $v_{9,2}v_{9,9}$ to $v_{9,1}v_{9,6}.$}
					\label{fig:fig1}
				\end{figure}

		We now present our examples. For these examples we will not apply Theorem~\ref{coh1}. We instead show most of our results by deriving an explicit form for generating function $g_n(z)$. We remark, however, that application of Theorem~\ref{coh1}, when appropriate, can provide a different expression for $\E$ and $\Var$ which may result in new identities for Catalan numbers in particular. For the first two examples, the results hold true for all convex polygons. For the rest of examples, we assume, in addition, the polygon is regular.
		
		\paragraph{Triangles with one side on $P_n$}
		One would ask how many of the triangles in the random triangulation $\calt_n$ have exactly one side in common with perimeter of $P_n.$ To answer this question we define $f_{n,l,j,r}$ as follows:
		\begin{align}\label{Eqfa2}
		f_{n,l,j,r} =
		\left\{
		\begin{array}{ll}
		1 & \mbox{if } l>1,\ \  j = l+1, \ \ r > j+1, \ \ r \leq n \\
		1 & \mbox{if } l>1,\ \  j> l+1,\  \ r = j+1,\ \  r\leq n \\
		1 & \mbox{if } 2 < j < n-1, \ \  l = 1,\ \  r = n \\	
		1 & \mbox{if } l = 1,\ \  j=2, \ \  3 < r < n \\	
		1 & \mbox{if } l = 1,\ \  j>2, \ \  r= j+1, \ \  r < n \\
		0 & \mbox{o.w.} \\
		\end{array}
		\right.
		\end{align}
		With this function, $\cals_{n}$ counts the number of triangles of interest. The following lemma provides some information for $\cals_n$.
		
		\begin{lemma} \label{oneside-thm}
			We have
			\begin{enumerate} [(I)]
				\item For all $n\geq4$,
				$$g_{n}(z)=\frac{1}{C_{n-2}}
				\sum_{j=0}^{n-2}C_j\left[2\binom{j+2}{n-2-j}-\binom{j+1}{n-2-j}\right]z^{2j+4-n}(1-z^2)^{n-2-j}.$$
				
				\item For all $n\geq4$, $\E(\cals_n) =\frac{n(n-4)}{2n-5}.$
				\item For all $n\geq 5$, $\Var(\cals_{n})=\frac{2n(n-1)(n-4)(n-5)}{(2n-5)^2(2n-7)}.$
			\end{enumerate}	
		\end{lemma}
		
		In the next result, we extend the previous example to slightly more general case where $f_n$ is define as
		\beqn \label{f-oneside-extension}
		f_{n,l,j,r} =\frac{1}{2}(w^{j-l}+w^{r-j}) \times \mbox{Eq} \mbox{ }\eqref{Eqfa2}.
		\feqn
		In particular, we have
		
		\begin{lemma} \label{oneside-gen-lemma}
		$$E(\cals_n)=-\frac{(n-1)(2w^{n-2}+3w)}{2(2n-5)}
			+\frac{3w}{C_{n-2}}\sum_{j=0}^{n-3}w^jC_j\binom{2n-6-2j}{n-3-j}
			-\frac{w}{2C_{n-2}}\sum_{j=0}^{n-3}w^jC_j\binom{2n-4-2j}{n-2-j},$$
			for all $n\geq4$.
		\end{lemma}
		
		We remark that by using simple identities
		\beq
		\sum_{j=0}^{n-3} C_j\binom{2n-6-2j}{n-3-j}=\binom{2n-5}{n-2} \mbox{ and } \sum_{j=0}^{n-3}C_j\binom{2n-4-2j}{n-2-j}=\binom{2n-3}{n-1}-C_{n-2}.
		\feq
		we can show this Lemma gives the same result when $w=1$ as Lemma \ref{oneside-thm}. This is another example

		\paragraph{Triangles with two sides on $P_n$ (Ears)}
		Next example is similar to the previous case with the exception that, in this example, we would ask how many of the triangles in $\calt_n$ have at least two sides residing on the perimeter of $P_n.$ To that end, we let $f_{n,l,j,r}$ to be as follows:
		\begin{align}\label{Eqfa1}
		f_{n,l,j,r} =
		\left\{
		\begin{array}{ll}
		1 & \mbox{if } 1\leq l , j = l+1, \ \ r = l+2, \ \ r < n \\
		1 & \mbox{if }  l=1, \ \ j = n-1,\ \  r=n\\
		1 & \mbox{if }	 l=1, \ \ j = 2, \ \ r=n \\
		0 & \mbox{o.w.} \\
		\end{array}
		\right.
		\end{align}
		Next lemma provides detailed information on $\cals_n$ which counts the number of triangles of interest in $\calt_n$:
		
		\begin{lemma}\label{twoside-thm}
			For all $n\geq4,$ we have
			\begin{enumerate}[(I)]
				\item
				$g_n(z)=1+\frac{1}{C_{n-2}}\sum_{j=0}^{n-3}C_j\left(\binom{j+1}{n-2-j}+2\binom{j+1}{n-3-j}\right)(z-1)^{n-2-j}.$\\
				\item $\E(\cals_n) = \frac{n(n-1)}{2(2n-5)}.$
				\item $\Var(\cals_{n})=\frac{n(n-1)(n-4)(n-5)}{2(2n-5)^2(2n-7)}$ for $n\geq6$.
			\end{enumerate}
		\end{lemma}
		
		\par
	
		Recall that there is a well-known bijection between binary trees with $n-2$ nodes and triangulations of $P_n$. See \cite{jesus1} for a review of various interesting bijections of similar nature. In \cite{noy3}, Hurtado and Noy use this bijection to give a combinatorial proof for section (I) and (II) of Lemma~\eqref{twoside-thm}. We remark that our method has the capability of generalizing this result to cases such as the  one described in \eqref{f-oneside-extension}, while it is not clear  how a combinatorial argument can provide such extension in a straightforward manner. Having the last two examples, one can also provide the exact distribution on the number of triangles with no side on the perimeter of $P_n$ also know as internal triangles. One final remark is that Lemma \ref{twoside-thm}-(II) and Lemma \ref{oneside-thm}-(II) imply that the average number of nodes with degree two (resp. one) in a uniformly sampled binary trees of $n-2$ nodes is $\frac{n(n-4)}{2n-5}$ (resp. $\frac{n(n-1)}{2(2n-5)}$).  See Figure~\ref{fig:fig5} for an example.
		\begin{figure}[h]
			\centering
			\begin{tabular}{c@{}c@{}}
				\includegraphics[page=1,width=0.45\textwidth]{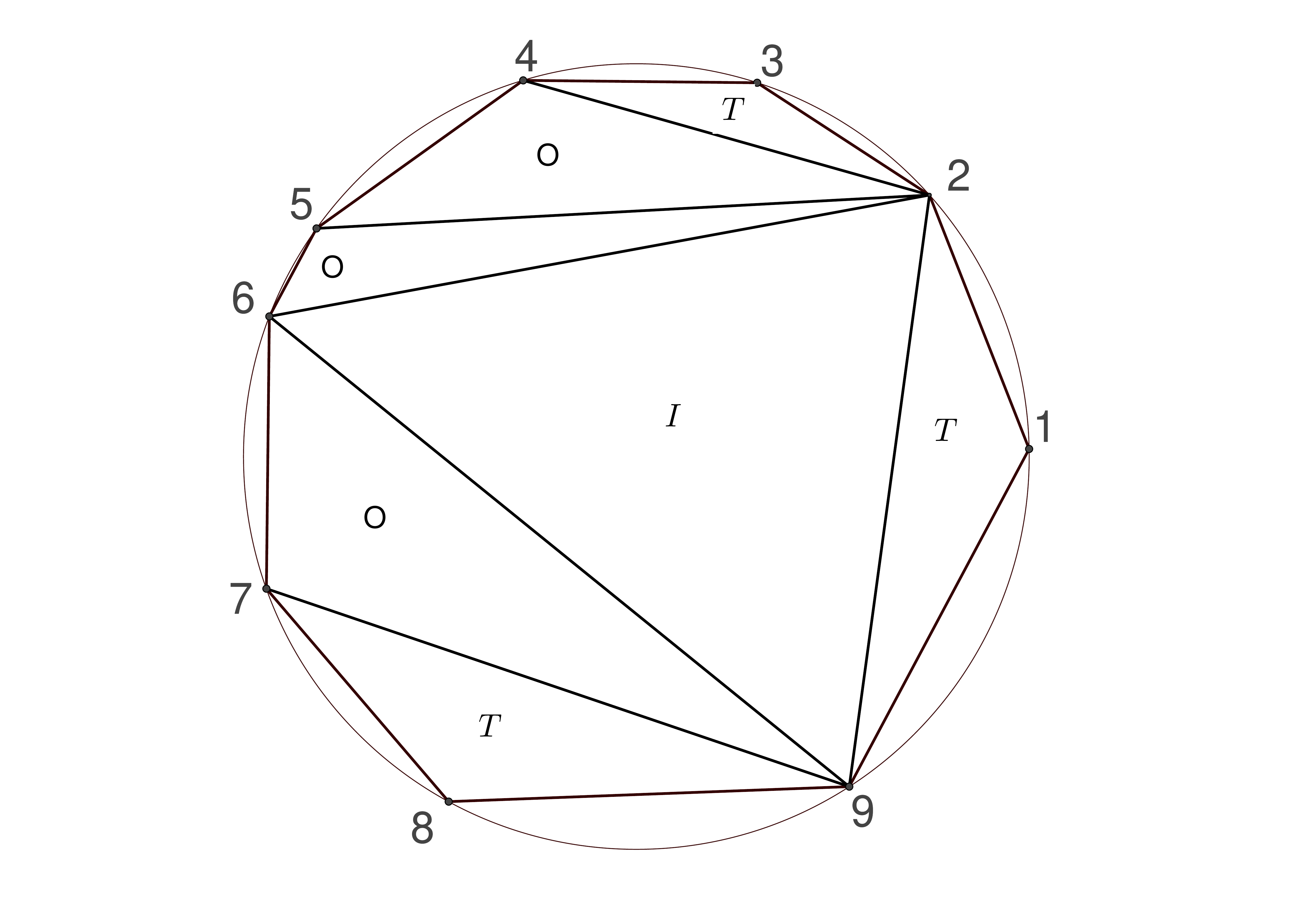} &
				\includegraphics[page=1,width=0.45\textwidth]{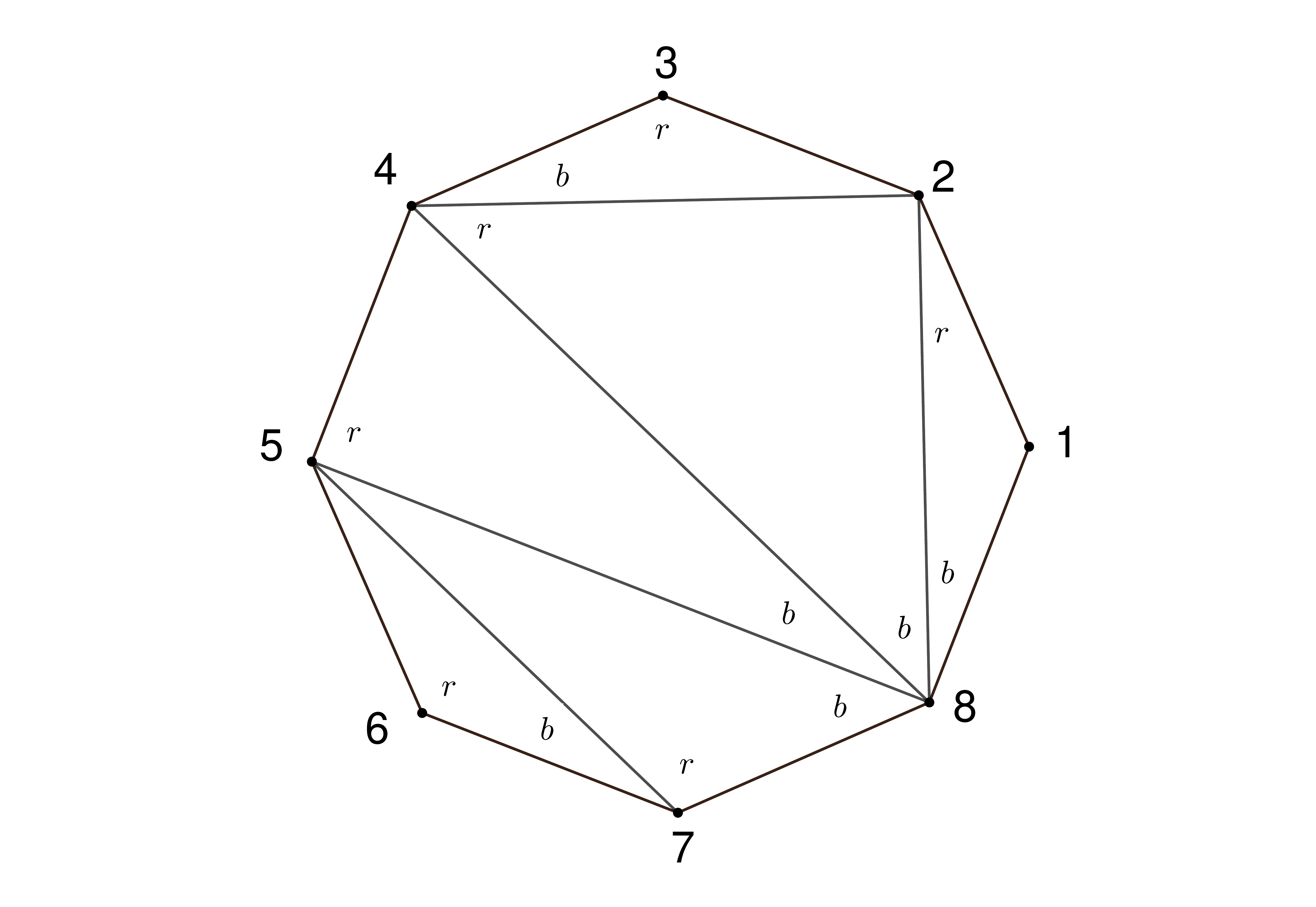} 	
			\end{tabular}
			\caption{(left) A triangulation of an irregular $P_9$. The triangles with one side (resp. two sides) on the perimeter of $P_n$ are marked by `O' (resp. `T'). There is also one internal triangle marked with $I$. (right) A triangulation of $P_8$ with marked angles }
			\label{fig:fig5}
		\end{figure}

		In the next few examples, we assume $v_{n,j} := (\cos\theta_{n,j}, \sin\theta_{n,j}),$ where $\theta_{n,i}:=\frac{2\pi (i-1)}{n}$ for $1\leq i \leq n.$ In other words, $P_n$ is a regular polygon inscribed in the unit circle.

		\paragraph{Degree of a vertex}
		Our objective in this example is to obtain some information on how a typical vertex of $\calt_n$ looks. Let $\cald_{n,i}$ be the number diagonals incident with $i$-th vertex in $\calt_n$. As it was shown in \cite{noy2}, any triangulation can be fully characterized by the sequence of degrees of the polygon vertices. Note that (a) by symmetry all $\cald_{n,i}$ have identical distributions. (b) $\sum_{i=1}^n \cald_{n,i} = 2(n-3)$. Therefore, we have $E(\cald_{n,1}) = \cdots = E(\cald_{n,n})=\frac{2(n-3)}{n}.$  By item (b), however, $\cald_{n,i}$ are dependent. Hence, in order to obtain the full description of $\cald_{n,1}$, we need to do a bit more work.  Note that Bernasconi et al. \cite{nicla1} provided an elegant means to study the vertices of $\calt_n$ in a very general sense. This is done by designing a Boltzmann sampler that reduces the study of $\cald_{n,i}$s to properties of sequences of independent and identical distributed random variables. At this point we are not able to extend their approach to our model, however, we believe that the proposed approach in \cite{nicla1} and \cite{flajolet2} might be proven to be useful in our case as well. 
		
		To that end, we let
		\beq \label{degree_fn}
		f_{n,l,j,r}=
		\left\{
		\begin{array}{ll}
			1 & \mbox{if } l=1 \\
			0 & \mbox{o.w.} \\
		\end{array}
		\right.
		\feq
		
		With this function, $\cals_n$ is indeed $\cald_{n,1}+1.$ Then, we get
		
		\begin{lemma} \label{degree-lem}
			For $n\geq 4$, we have
			\begin{enumerate}
				\item $g_{n}(z)=\frac{1}{C_{n-2}}\sum_{s=1}^{n-2}\frac{s(2n-s-5)!}{(n-s-2)!(n-2)!}z^s.$
				\item $\E(\cals_{n})=\frac{3(n-2)}{n}$ for $n\geq2$.
				\item $\Var(\cals_{n})=\frac{2(2n-3)(n-2)(n-3)}{n^2(n+1)}$ for $n\geq2$.				
			\end{enumerate}
		\end{lemma}
		
		In addition to \cite{nicla1}, Devroye et al. \cite{flajolet1} also studied the maximum of this sequence namely $\max_{1\leq i \leq n} \cald_{n,i}$ where they obtained same result for $\cald_{n,1}$ (See Lemma~1 of \cite{flajolet1}). Their proof is purely combinatorial while ours is based on derivation of the generating function $g_n(z)$.
		
		Our main result for this example is to characterize the distribution of the portfolio of angles at the vertex $1$. More precisely,
		
		\begin{theorem} \label{degree-thm}
		Let $\cala_{n,i}$ be the number of angles of size $\frac{2\pi i}{n}$ at vertex $1$ of $\calt_n$. Then, for a fix sequence $0\leq k_1,\cdots,k_{n-2} \leq n-2$ with $ \sum_{i=1}^{n-2} i k_i = n-2$, we have
		\beq
		\PP(\cala_{n,1} = k_1, \cdots, \cala_{n,n-2} = k_{n-2}) = \frac{K(2n-K-5)!}{Z_{n,K} C_{n-2}(n-K-2)!(n-2)!}\binom{K}{k_1,\cdots, k_{n-2}} C_{0}^{k_1}C_{1}^{k_2}\cdots C_{n-3}^{k_{n-2}},
		\feq
		where $K:=\sum_{i=1}^{n-2} k_i$ and
		\beqn \label{Znk-def}
		Z_{n,K} := \sum_{\substack{\sum_{j=1}^{n-2} jp_j = n-2\\ \sum_{j=1}^{n-2}p_j=K \\ 0 \leq p_j \leq n-2}} \binom{K}{p_1,\cdots, p_{n-2}} C_{0}^{p_1}\cdots C_{n-3}^{p_{n-2}}.
		\feqn
		\end{theorem}
		
		\iffalse
		\paragraph{Right triangles}
		It is easy to see that if $n$ is odd, then no triangulation has a right triangle. Additionally,
		
		\begin{lemma}
			Suppose $n$ is even. The probability that $\calt_n$ has two right triangles is $$r_n:=\frac{nC_{\frac{n}{2}-1}^2}{2C_{n-2}},$$ and the probability that $\calt_n$ has no right triangle is $1-r_n.$
		\end{lemma}
		
		This follows from the following simple argument. If $n$ is even, then in order for a triangulation of $P_n$ to have a right triangle, it needs to have a diagonal of form $v_{n,i}v_{n,i+\frac{n}{2}}$. For a fixed $i$, the number of such triangulations is $C_{\frac{n}{2} - 1}^2$. In each of these triangulations, we always have two right triangles, one on each side of the diagonal $v_{n,i}v_{n,i+\frac{n}{2}}$. Letting $i$ to take different values between $1$ and $\frac{n}{2},$ the total of number of triangulations with two right triangles is $\frac{n}{2}C^2_{\frac{n}{2}-1}.$ This completes the argument.
        \fi

		\paragraph{Blue angles}
		Suppose for all $1\leq l<j<r\leq n$ we mark the triangle $\Delta_{n,l,j,r}$ such that $\angle v_{n,l}v_{n,j}v_{n,r}$ is red, $\angle v_{n,j}v_{n,l}v_{n,r}$ is green, and $\angle v_{n,j}v_{n,r}v_{n,l}$ is blue. In the next two examples we focus on various properties of blue angles. Similar results hold for the other two colors by symmetrical arguments therefore we will not present them. See Figure~\ref{fig:fig5} for an example on how the marking process works. We note that the total sum of blue angles in $\calt_n$ can be studied by defining
		\beqn \label{sum_blue-eq}
		f_{n,l,j,r}=j-l.
		\feqn
		Then it is easy to show
		\begin{lemma} \label{sumblue-lemma}
		$\E(\cals_{n})=\frac{2^{2n-5}-\binom{2n-5}{n-2}}{C_{n-2}}.$
		\end{lemma}
		
		Next, we count the number of blue angles equal to $\frac{2\pi p}{n}$ for a fixed $1\leq p \leq n-1.$ To that goal, we define
		\beqn \label{blue-p-eq}
		f_{n,l,j,r}=
		 \left\{
		 \begin{array}{ll}
		 1 & \mbox{if } j-l=p \\
		 0 & \mbox{o.w.} \\
		 \end{array}
		 \right.
		\feqn
		Here, we only report the result for $p=1$ and leave the general case to reader with an understanding the general case follows from the same argument with a slight modification in the initial conditions.
		
		\begin{theorem} \label{blue-p-thm}
			Fix $p=1$. For $n\geq 4$, we have
			\begin{enumerate}
				\item $g_{n}(z)=\frac{1}{C_{n-2}}\sum_{j=1}^{n-2}N_{n-2,j}z^j,$
				where $N_{n,k}:=\frac{1}{n}\binom{n}{k}\binom{n}{k-1}$s are Narayana numbers.
				\item $\E(\cals_{n})=\frac{n-1}{2}$.
				\item $\Var(\cals_{n})=\frac{(n-1)(n-2)(n-3)}{2(2n-5)}$.					
			\end{enumerate}
		\end{theorem}
	
		For more information on Narayana numbers, see the sequence A001263 in \cite{sloane1} and Exercise 6.36 in \cite{stanley1}.

		\iffalse

		\paragraph{A curious example} For our last example, we give a somewhat different function where $f_n$ is a polynomial. More precisely, for a fixed $w\in \cc,$ we define
		\beqn
		f_{n,l,j,r} = \frac{1}{3}(w^{j-l} + w^{r-j} + w^{r-l}).
		\feqn
		With this function,
		
		\begin{lemma}\label{crious-lemma}
		For all $n\geq4$
		$$\E(\cals_n)=w^{n-1}-\frac{(n-1)w}{3}+\frac{2w}{3C_{n-2}}\sum_{j=0}^{n-3}w^jC_j\binom{2n-4-2j}{n-2-j}.$$
		\end{lemma}
		
		We remark that in this case combinatorial arguments seem to fail to be applicable and currently we are only cable of obtaining the result using this paper approach. \\ \par
		
		\fi

		This paper is organized as follows. In Section \ref{sec2} we introduced the main tools and prove Theorem \ref{coh1}. Section \ref{examples-sector} includes the proof of results for the examples.
		
		\section{An algorithm and structure of $g_n(z)$} \label{sec2}
		\par
		We begin this section with describing an algorithm that generates a uniformly sampled random triangulation of $P_n.$ We note that are currently various paradigms in the literature for sampling of a random triangulation. We refer to \cite{flajolet1} and \cite{epstein1} for algorithmic instances, to \cite{tetali2, mol1, tetali1} for random walk based samplers, and to \cite{nicla1} and \cite{flajolet2} for Boltzmann samplers.  Due to its constructive recursive nature, we choose the following simple algorithm belonging to the community folklore. For a given $1\leq l<r\leq n,$ we define the function $\mu_{n,l,r}$ such that
		\beqn \label{mu-def}
		\mu_{n,l,r}(j) = \frac{C_{j-l-1}C_{r-j-1}}{C_{r-l-1}}.
		\feqn
		Note $\mu_{n,l,r}$ is indeed a probability distribution on integer numbers between $l$ and $r$ since by Catalan recursive identity we have
		\beq\label{eqcat}
		C_0 = 1, \quad \mbox{and} \quad C_{m+1} = \sum_{s=0}^{m} C_s C_{m-s} \ \ \mbox{for}\ \  m\geq 0.
		\feq
		Next, we define our sampling algorithm. With an abuse of notation we refer to this algorithm also as $\calt_{n,l,r}.$ It should be clear from the context whether we intend the algorithm or the triangulation itself.
		
		\paragraph{Sampling algorithm: $\calt_{n,r,l}$}
		\begin{enumerate}
			\item Generate random integer $\calj=j$ between $l$ and $r$ with probability $\mu_{n,l,r}$.
			\item If $r>l+2,$ then return $\calt_{n,l,j} \cup \calt_{n,j,r} \cup \Delta_{n,l,j,r}$.
			\item If $r=l+2,$ then return $\Delta_{n,l,l+1,l+2}.$
			\item If $r<l+2,$ then return empty.
		\end{enumerate}
		Note that for each fixed triangle $\Delta_{n,l,j,r}$ there are exactly $C_{l-j-1}C_{r-j-1}$ triangulations with $\Delta_{n,l,j,r}$ among their triangles. Therefore, the probability that a uniformly sampled triangulation from $T_{n,l,r}$ has the triangle $\Delta_{n,l,j,r}$ is exactly $C_{l-j-1}C_{r-j-1}C_{r-l-1}^{-1}.$ Given that $\calt_{n,l,j}$ and $\calt_{n,j,r}$ are independent, an inductive argument implies that $\calt_n$ is uniformly distributed on $T_n.$
		
		We are now ready to study $g_n(z)$ as the main tool in this paper. To that end, we note that by the algorithm $\calt_{n,l,r}$, we have
		\beqn \label{Sn-recur}
		\cals_{n,l,r} = f_{n,l,\calj,r} + \cals_{n,l,\calj} + \cals_{n,\calj,r},
		\feqn
		for $r,l\in [n]$ with $r-l>2.$ Similarly,
		\beqn \label{Sn-init}
		\cals_{n,l,l+2} = f_{n,l,l+1,l+2}, \quad \mbox{and} \quad \cals_{n,l,l+1} = 0.
		\feqn
		Recall \eqref{gnrl-def} and \eqref{mu-def}. Define $h_{n,l,r}(z)=C_{r-l-1}g_{n,l,r}(z)$. By the recursive equations \eqref{Sn-recur} and \eqref{Sn-init}, we have
		\beqn \label{EqT3}
		h_{n,l,r}(z) =\sum_{j=l+1}^{r-1} z^{f_{n,l,j,r}} h_{n,l,j}(z)h_{n,j,r}(z)\mbox{ with }h_{n,l,l+1}(z)=1,\,h_{n,l,l+2}(z)=z^{f_{n,l,l+1,l+2}}.
		\feqn
		
		We first give the following lemma \ref{lem1} that indicates, for a certain  class of functions $f_n,$ rotation and shifts do not effect the form of $h_{n,l,r}(z)$.

		\begin{lemma}\label{lem1}
			Suppose $f_{n,l,j,r}$ is a function of $r-j,$ $r-l,$ $j-l$ and possibly $n$. Then
			\begin{enumerate}[(I)]
				\item For all $1\leq l<r\leq n-1$, $h_{n,l,r}(z)=h_{n,l+1,r+1}(z)$.
				\item Suppose $n\geq 4.$ Additionally, assume $f_{n,l,j,r}$ is independent of $n$. Then $$h_{n,l,n}(z)=h_{n-1,l-1,n-1}(z)$$ for all $l=3,4,\ldots,n-1$.
			\end{enumerate}
		\end{lemma}
		\begin{proof}[Proof of (I)]
			Since $f_{n,l,j,r}$ is merely a function of $r-j,$ $r-l,$ $j-l$ and possibly $n$, we have that $f_{n,l+1,j+1,r+1}=f_{n,l,j,r}$ for all $1\leq l<r\leq n-2$. We proceed the proof by induction on $p:=r-l$, that is, we show that $h_{n,l,l+p}(z)=h_{n,l+1,l+p+1}(z)$ for all $1\leq p \leq n-1$ with an understanding that $1\leq l < l+p \leq n$. By \eqref{EqT3}, we have that $h_{n,l,l+1}(z)=h_{n,l+1,l+2}(z)=1$ and $h_{n,l,l+2}(z)=h_{n,l+1,l+3}(z)$, which implies that the lemma hold for $p=1,2$. Next, we assume that the lemma holds for $p=1,2,\cdots, s-1$ and prove it also holds for $p=s$. In other words, we show $h_{n,l+1, l+s+1}(z) = h_{n,l, l+s}(z)$. To that end, by \eqref{EqT3}, we obtain
			\begin{align*}
			h_{n,l+1,l+s+1}(z)&=\sum_{j=l+2}^{l+s}z^{f_{n,l+1,j,l+s+1}}h_{n,l+1,j}(z)h_{n,j,l+s+1}(z)\\
			&=\sum_{j=l+2}^{l+s}z^{f_{n,l,j-1,l+s}}h_{n,l,j-1}(z)h_{n,j-1,l+s}(z) \\
			&=\sum_{j=l+1}^{l+s-1}z^{f_{n,l,j,l+s}}h_{n,l,j}(z)h_{n,j,l+s}(z) =	h_{n,l,l+s}(z)	\end{align*}
			where for the second equality we used the induction hypothesis.
		\end{proof}
		\begin{proof}[Proof of (II)]
			By the assumption $f_{n,s,j,n} = f_{n-1,s-1,j-1,n-1}$ for all $1<s<j\leq n$. We proceed the proof by induction on $l=n-1,n-2,\ldots,3$. By \eqref{EqT3}, we have that $h_{n,n-1,n}(z)=h_{n-1,n-2,n-1}(z)=1$ and $h_{n,n-2,n}(z)=h_{n-1,n-3,n-1}(z)$, which shows that the claim holds for $l=n-1,n-2$. We assume that the claim holds for $l=n-1,n-2,\ldots,s+1$ and show that it also holds for $l=s$. By \eqref{EqT3} and Lemma \ref{lem1}-(I), we have
			\beq
			h_{n,s,n}(z)&=&\sum_{j=s+1}^{n-1}z^{f_{n,s,j,n}}h_{n,s,j}(z)h_{n,j,n}(z)= \sum_{j=s+1}^{n-1}z^{f_{n,s,j,n}} h_{n,s+n-j,n}(z)h_{n,j,n}(z),
			\feq
			and
			\beq
			h_{n-1,s-1,n-1}(z)&=&\sum_{j=s}^{n-2}z^{f_{n-1,s-1,j,n-1}} h_{n-1,s-1,j}(z)h_{n-1,j,n-1}(z) \\
			&=&\sum_{j=s}^{n-2}z^{f_{n-1,s-1,j,n-1}}h_{n-1,n-j+s-2,n-1}(z)h_{n-1,j,n-1}(z)\\
			&=&\sum_{j=s+1}^{n-1}z^{f_{n,s,j,n}}h_{n-1,n-j+s-1,n-1}(z)h_{n-1,j-1,n-1}(z) \\
			&=& \sum_{j=s+1}^{n-1}z^{f_{n,s,j,n}} h_{n,s+n-j,n}(z)h_{n,j,n}(z).
			\feq
			Where we used the induction hypothesis for the last equality. Therefore, we have shown $h_{n,s,n}(z)=h_{n-1,s-1,n-1}(z)$, which completes the induction.
		\end{proof}

		 Recall that $h_{n}(z)=C_{n-2}g_{n}(z).$ Therefore, $\E$ and $\Var$ follow from $h_n(z)$:
		 \beqn
		 \E(\cals_n) = \frac{1}{C_{n-2}}h'_{n}, \mbox{ and }
		 \Var(\cals_n) = \frac{1}{C_{n-2}}(h'_{n}+h''_{n}) - \frac{1}{C^2_{n-2}}(h'_{n})^2 \label{var-e-form},
		 \feqn
		 where
		 \beq \label{hn_prime}
		 h'_n:=\frac{d}{dz}h_{n}(z)\mid_{z=1} \mbox{ and } h''_n:=\frac{d^2}{dz^2}h_{n}(z)\mid_{z=1}. \\		
		 \feq
		 Similarly, we define $h'_{n,l,r}$ and $h''_{n,l,r}$.
		
		\begin{proof}[Proof of Theorem~\ref{coh1}]
		Suppose $f_{n,l,j,r}$ is a function of $r-j,$ $r-l,$ $j-l$ and possibly $n$. We will calculate $h'_{n,l,n}$ and $h''_{n,l,n}$ to prove Theorem~\ref{coh1}. To that end, note that Lemma \ref{lem1}-(I) reduces the calculation $h_{n,l,r}(z)$ to that of $h_{n,l,n}(z)$. In other words, equation \eqref{EqT3} is reduced to
		\beqn\label{rech1}
		& h_{n,l,n}(z) =\sum_{j=l+1}^{n-1} z^{f_{n,l,j,n}} h_{n,l+n-j,n}(z)h_{n,j,n}(z)\notag \\ &\mbox{with} \  h_{n,n-1,n}(z)=1, \ h_{n,n-2,n}(z)=z^{f_{n,n-2,n-1,n}}.
		\feqn
		\par		
		By $h_{n,l,n}(1)=C_{n-l-1}$, we rewrite \eqref{rech1} as
		\beqn \label{hprime-rec}
		h'_{n,l,n}(z) =\sum_{j=l+1}^{n-1} f_{n,l,j,n} C_{j-l-1}C_{n-j-1} +\sum_{j=l+1}^{n-1}(h'_{n,l+n-j,n}C_{n-j-1}+C_{j-l-1}h'_{n,j,n})
		\feqn
		with $h'_{n,n-1,n}=0$ and $h'_{n,n-2,n}=f_{n,n-2,n-1,n}$. Now, define ${\bf M}_n$ to be the matrix $(m_{ij})_{1\leq i,j\leq n-1}$ where
		\beq
		m_{ij} = \left\{
		\begin{array}{ll}
			1  & \mbox{if }1\leq i=j \leq n-1 \\
			-2C_{j-i-1} & \mbox{if } 1\leq i<j\leq n-1 \\
			0  & \mbox{if } 1\leq j<i\leq n-1.
		\end{array}
		\right.
		\feq
		Recall \eqref{beta_nl}. Then, the recurrence \eqref{hprime-rec} can be written as
		\begin{align}\label{eqMh1}
		{\bf M}_n(h'_{n,1,n},\ldots,h'_{n,n-1,n})^T=(\beta_{n,1},\ldots,\beta_{n,n-2},0)^T.
		\end{align}	
		To solve this system of equations, we define the matrix
		${\bf D}_n=(d_{ij})_{1\leq i,j\leq n-1}$, where
		\beq
		d_{ij} = \left\{
		\begin{array}{ll}
			\binom{2j-2i}{j-i}  & \mbox{if } 1\leq i\leq j\leq n-1\\
			0 & \mbox{if } 1\leq j<i\leq n-1.
		\end{array}
		\right.
		\feq
		Recall the generating function of Catalan numbers:
		\beqn\label{C_gf}
		C(t)=\sum_{n\geq0}C_nt^n=\sum_{n\geq0}\frac{1}{n+1}\binom{2n}{n}t^n=\frac{1-\sqrt{1-4t}}{2t}.
		\feqn
		Since the matrices ${\bf M}_n$ and ${\bf D}_n$ are upper triangular with diagonal ones, we have that $\sum_{j=1}^{n-1}m_{ij}d_{jl}=0$ for all $1\leq l<i \leq n-1$ and $\sum_{j=1}^{n-1}m_{ij}d_{ji}=1$ for all $1\leq i \leq n-1$. Suppose $1\leq i< l \leq n-1.$ We observe that from the convolution $$\frac{1}{2x\sqrt{1-4x}}-\frac{1}{2x}=C(x)\times\frac{1}{\sqrt{1-4x}}=\sum_{n\geq0}\sum_{j=0}^nC_j\binom{2j}{j}x^n,$$
			we obtain
			$$2\sum_{j=0}^nC_j\binom{2s-2j}{s-j}=\binom{2s+2}{s+1}.$$
			Hence, $\sum_{j=1}^{n-1}m_{ij}d_{jl}=0$ for all $1\leq i<l\leq n-1$.
		This shows that for all $n\geq2$, ${\bf M}_n{\bf D}_n=I_{n-1}$, where $I_n$ is the $(n\times n)$ identity matrix.		
		Similarly, \eqref{rech1} implies
		\begin{align}
		h'' _{n,l,n} &=\sum_{j=l+1}^{n-1}f_{n,l,j,n}(f_{n,l,j,n}-1)C_{j-l-1}C_{n-j-1}
		+\sum_{j=l+1}^{n-1}f_{n,l,j,n}(h'_{n,l+n-j,n}C_{n-j-1}+h'_{n,j,n}C_{j-l-1})\notag\\
		&+\sum_{j=l+1}^{n-1}(h''_{n,l+n-j,n}C_{n-j-1}+2h'_{n,l+n-j,n}h'_{n,l,n}+h''_{n,j,n}C_{j-l-1})\label{eqhh2}
		\end{align}
		with $h''_{n,n-1,n}=0$ and $h''_{n,n-2,n}(z)=f_{n,n-2,n-1,n}(f_{n,n-2,n-1,n}-1)$. With notation
		\beq
		\gamma_{n,l}=\sum_{j=l+1}^{n-1}\biggl(&f_{n,l,j,n}(f_{n,l,j,n}-1)C_{j-l-1}C_{n-j-1}\notag\\
		&+2f_{n,l,j,n}(h'_{n,l+n-j,n}C_{n-j-1}+h'_{n,j,n}C_{j-l-1})
		+2h'_{n,l+n-j,n}h'_{n,l,n}\biggr),\label{eqgamma}
		\feq
		\eqref{eqhh2} can be written as
		$${\bf M}_n(h''_{n,1,n},\ldots,h''_{n,n-1,n})^T=(\gamma_{n,1},\ldots,\gamma_{n,n-2},0)^T.$$
		By \eqref{eqMh1} and the fact that ${\bf M}_n{\bf D}_n=I_{n-1}$ for $n\geq 2$, we obtain
		$$(h'_{n,1,n},\ldots,h'_{n,n-1,n})^T={\bf D}_n(\beta_{n,1},\ldots,\beta_{n,n-2},0)^T$$
		and
		$$(h''_{n,1,n},\ldots,h''_{n,n-1,n})^T={\bf D}_n(\gamma_{n,1},\ldots,\gamma_{n,n-2},0)^T.$$
		Thus, for all $l=1,2,\ldots,n-2$,
		$$h'_{n,l,n}=\sum_{j=l}^{n-2}\beta_{n,j}\binom{2j-2l}{j-l}\mbox{ and }h''_{n,l,n}=\sum_{j=l}^{n-2}\gamma_{n,j}\binom{2j-2l}{j-l},$$
		which complete the proof of Theorem~\ref{coh1}.
		\end{proof}
		
		\begin{example}	
		If $f_{n,l,j,r}=1$ for all $1\leq l< j < r\leq,$ then
		$$\beta_{n,j}=\sum_{i=j+1}^{n-1}C_{i-j-1}C_{n-i-1}=\sum_{i=0}^{n-2-j}C_iC_{n-2-j-i}=C_{n-1-j},$$
		which leads to
		\beq
		C_{n-2}E({\cals_{n}})&=&\sum_{j=1}^{n-2}\beta_{n,j}\binom{2j-2}{j-1}\\
		&=&\sum_{j-1}^{n-2}C_{n-1-j}\binom{2j-2}{j-1}=\binom{2n-3}{n-2}-\binom{2n-4}{n-2}
		=(n-2)C_{n-2},
		\feq
		as expected.
		\end{example}
		\begin{example}
		Suppose $f_n$ is a polynomial, where for a fixed $w\in \cc,$ 
		\beq
		f_{n,l,j,r} = \frac{1}{3}(w^{j-l} + w^{r-j} + w^{r-l}).
		\feq
		Then, for all $n\geq4,$ Theorem~\ref{coh1} implies
		$$\E(\cals_n)=w^{n-1}-\frac{(n-1)w}{3}+\frac{2w}{3C_{n-2}}\sum_{j=0}^{n-3}w^jC_j\binom{2n-4-2j}{n-2-j}.$$
		\end{example}
		\section{Examples} \label{examples-sector}
		
		The main idea for all the proofs in this section is as follows. We define two generating functions $H_2(t,z)=\sum_{n\geq3}h_{n,2,n}(z)t^{n-3}$ and  $H_1(t,z)=\sum_{n\geq3}h_{n}(z)t^{n-3}.$ Our end goal is to obtain $H_1$ as $h_n(z)$ can be easily obtained by extracting the coefficients of $t^{n-3}$. However, in most cases, we first obtain $H_2$ and then solve $H_1$ with respect to $H_2.$ To that end, we first simplify \eqref{EqT3} using certain properties of $f_n$ at hand and then derive explicit equations for $H_1$ and $H_2$ through the application of recursion \eqref{EqT3}.
		
		\subsection{Triangles with only one side on $P_n$}\label{seconeside}
		
		In this subsection, we give the proof of Theorem~\ref{oneside-thm}. We Recall \eqref{Eqfa2}. Note that by \eqref{EqT3} and Lemma \ref{lem1}, we have $$h_{n,2,n}(z)=2zh_{n-1,2,n-1}(z)+\sum_{j=4}^{n-2}h_{j,2,j}(z)h_{n-j+2,2,n-j+2}(z),$$
		with $h_{3,2,3}(z)=h_{4,2,4}(z)=1$. Multiplying  by $t^{n-3}$ and summing over $n\geq5$, we obtain
		$$H_2(t,z)-t-1=2zt(H_2(t,z)-1)+t(H_2(t,z)-1)^2.$$
		By solving this equation, we obtain
		\beqn
		H_2(t,z)=\frac{1+2(1-z)t-\sqrt{1-4t(z-z^2t+t)}}{2t}=1-z+(z+t-z^2t)C(t(z+t-z^2t)). \label{pro1B}
		\feqn
		Thus, by \eqref{C_gf}, for all $n\geq4$, $$h_{n,2,n}(z)=\sum_{j=0}^{n-3}C_j\binom{j+1}{n-3-j}z^{2j+4-n}(1-z^2)^{n-3-j}.$$
		By \eqref{EqT3} with using Lemma \ref{lem1}, we have $$h_{n,1,n}(z)=2h_{n,2,n}(z)+z\sum_{j=3}^{n-2}h_{j+1,2,j+1}(z)h_{n-j+2,2,n-j+2}(z)$$
		with $h_{3,1,3}(z)=1$. By multiplying  by $t^{n-3}$ and summing over $n\geq4$, we obtain	$$H_1(t,z)=1+2(H_2(t,z)-1)+z(H_2(t,z)-1)^2.$$
		By \eqref{pro2B}, we obtain
		\beqn \label{pro2B}
		H_1(t,z)&=\frac{2tz^3-3tz-z^2+t}{t} -\frac{(2tz^2-2t-z)(z+t-z^2t)}{t}C\left(t(z+t-z^2t)\right).
		\feqn
		where $C(.)$ is defined by \eqref{C_gf}. Hence, for all $n\geq4$, $$h_{n}(z)=
		\sum_{j=0}^{n-2}C_j\left[2\binom{j+2}{n-2-j}-\binom{j+1}{n-2-j}\right]z^{2j+4-n}(1-z^2)^{n-2-j}.$$
		This finishes the proof of Theorem~\ref{oneside-thm}-(I).
		\iffalse
		\begin{lemma}\label{mth1B}
			For all $1\leq l<r\leq n$, $h_{n,l,r}(z)=C_{r-l-1}g_{n,l,r}(z)$ is the coefficient of $w^{l+n-r}t^n$ in $H(t,w,z)$, where
			\begin{align*}
			H(t,w,z)&=wt^3H_1(t,z)+\frac{w^2t^3}{1-wt}H_2(t,z),
			\end{align*}
			and $H_2(t,z),H_1(t,z)$ are given in \eqref{pro1B} and \eqref{pro2B}.
		\end{lemma}
		
		By Lemma \ref{mth1B}, we obtain the first values of $g_{n,l,r}(z)$ as showed in Table \ref{tab1B}.
		\fi
		
		Next, by \eqref{pro2B}, we have
		$$H_1'(t,1):=\frac{\partial}{\partial z}H_1(t,z)\mid_{z=1}				=\frac{t}{2}-4t^2+3t^3-\frac{t(20t^2-10t+1)}{2\sqrt{1-4t}}.$$
		The coefficient of $t^{n-3}$ in $H_1'(t,1)$ is
		$$h'_n =\frac{-1}{2}\binom{2n-2}{n-1}+5\binom{2n-4}{n-2}-10\binom{2n-6}{n-3}.$$
		This completes the Proof of Theorem~\ref{oneside-thm}-(II). Similarly,  \eqref{pro2B} gives
		$$\frac{\partial^2}{\partial z^2}H_1(t,z)\mid_{z=1}	=12t^3-4t^2+\frac{4t^2(1-9t+24t^2-14t^3)}{(1-4t)^{3/2}},$$
		Extracting the coefficient of $t^{n-3}$ gives $$\frac{1}{C_{n-2}}h''_n
		=\frac{n(n-1)(n^2-9n+20)}{(2n-5)(2n-7)}.$$
		Therefore, $\Var(\cals_n)$ is followed from \eqref{var-e-form}.
		
		\paragraph{A slight generalization}
		By similar arguments as in the beginning of this section, we have
		\begin{align*}
		h_{n,2,n}(z)&=(z^w+z^{w^{n-3}})h_{n-1,2,n-1}(z)+\sum_{j=4}^{n-2}h_{j,2,j}(z)h_{n-j+2,2,n-j+2}(z),\\
		h_{n,1,n}(z)&=2h_{n,2,n}(z)+\sum_{j=3}^{n-2}z^{w^{j-1}}h_{j+1,2,j+1}(z)h_{n-j+2,2,n-j+2}(z)
		\end{align*}
		with $h_{4,2,4}(z)=h_{3,1,3}(z)=1$. Differentiating $h_n(z)$ at $z=1$ and using the fact $h_{n,l,r}(1)=C_{r-l-1}$, we obtain
		\begin{align*}
		h'_{n,2,n}&=(w+w^{n-3})C_{n-4}+2h'_{n-1,2,n-1}+2\sum_{j=4}^{n-2}h'_{j,2,j}C_{n-j-1},\\
		h'_{n,1,n}&=2h'_{n,2,n}+\sum_{j=3}^{n-2}w^{j-1}C_{j-2}C_{n-j-1}
		+2\sum_{j=3}^{n-2}h'_{j+1,2,j+1}C_{n-j-1}
		\end{align*}
		with $h'_{4,2,4}=h'_{3,1,3}=0$.
		
		Define $H_1'(t)=\sum_{n\geq3}h'_{n,1,n}t^{n-3}$ and $H_2'(t)=\sum_{n\geq3}h'_{n,2,n}t^{n-3}$. Then, the above recurrences can be rewritten in terms of $H_1'(t)$ and $H_2'(t)$ as
		\begin{align*}
		H'_2(t)&=wt(C(t)-1)+wt(C(wt)-1)+2tH'_2(t)C(t),\\
		H_1'(t)&=w(C(wt)-1)(C(t)-1)+2H'_2(t)C(t).
		\end{align*}
		Thus,
		$H'_2(t)=\frac{wt(C(t)+C(wt)-2)}{\sqrt{1-4t}}$ and \begin{align*}
		H_1'(t)&=w(C(wt)-1)(C(t)-1)+\frac{2wtC(t)(C(t)+C(wt)-2)}{\sqrt{1-4t}}\\
		&=\frac{w(1-4t)C(wt)}{2t}+\frac{w(6t-1)C(wt)}{2t\sqrt{1-4t}}-3w(C(t)-1).
		\end{align*}
		Recall the generating function \eqref{C_gf} and $\frac{1}{\sqrt{1-4t}}=\sum_{n\geq0}\binom{2n}{n}t^n$. Thus, by extracting the coefficient of $t^{n-3}$ in $H_1'(t)$ and by \eqref{var-e-form}, we complete the proof of Lemma~\ref{oneside-gen-lemma}.
		
		\iffalse
		Note that by Proposition \ref{coh1}, we have
		$$\E(\cals_n)=\frac{1}{C_{n-2}}\sum_{j=1}^{n-2}\left(\sum_{i=j+1}^{n-1} f_{n,j,i,n}C_{i-j-1}C_{n-i-1}\right)\binom{2j-2}{j-1}.$$
		As a corollary, we get the following identity,
		\begin{align*}
		\sum_{j=0}^{n-3}w^{j+1}C_j\left(6\binom{2n-6-2j}{n-3-j}
		-\binom{2n-4-2j}{n-2-j}\right)-2(2w^{n-2}+3w)C_{n-3}
		=\sum_{j=1}^{n-2}\left(\sum_{i=j+1}^{n-1} f_{n,j,i,n}C_{i-j-1}C_{n-i-1}\right)\binom{2j-2}{j-1}.
		\end{align*}
		\fi
		
		\subsection{Triangles with two sides on $P_n$}
		
		In this subsection, we give the proof of Theorem~\ref{twoside-thm}. The proof is very similar to that of the previous section. Note that $h_{3,2,3}(z)=1$ and $h_{4,2,4}(z)=z$. By \eqref{EqT3}, \eqref{Eqfa1}, and Lemma \ref{lem1}, for $n\geq5$ we have
		\beqn
		h_{n,2,n}(z)=\sum_{j=3}^n h_{j,2,j}(z)h_{n-j+2,2,n-j+2}(z).
		\feqn
		Multiplying by $t^{n-3}$ and summing over all terms, we obtain
		$$H_2(t,z)-zt-1=-t+t(H_2(t,z))^2.$$
		Equivalently,
		\beqn \label{twoside_H2}
		H_2(t,z)=\frac{1-\sqrt{1-4t+4(1-z)t^2}}{2t}.
		\feqn
		Once again, \eqref{EqT3} and Lemma \ref{lem1} imply
		\beqn
		h_{n,1,n}(z)=2zh_{n,2,n}(z)+\sum_{j=4}^{n-1}h_{j,2,j}(z)h_{n-j+3,2,n-j+3}(z),
		\feqn
		with $h_{3,1,3}(z)=z$ and $h_{4,1,4}(z)=2z^2$. Multiplying  by $t^{n-3}$ and summing over $n\geq5$, we obtain				$$H_1(t,z)-2z^2t-z=2z(H_2(t,z)-zt-1)+(H_2(t,z)-1)^2.$$
		Solving for $H_1(t,z),$ and replacing $H_2(z,t)$ from \eqref{twoside_H2}, we have
		\beq \label{twoside-H1} H_1(t,z)=\frac{1-2(2-z)t+4(1-z)t^2-(1-2(1-z)t)\sqrt{1-4t+4(1-z)t^2}}{2t^2}.
		\feq
		
		Recall the generating function \eqref{C_gf}. To extract the coefficients of $H_1,$ we rewrite $H_1$ using $C(t)$ function as follows
		\beqn \label{twoside-H1sol}
		H_1(t,z)&=&\frac{(1-2(1-z)t)(1-(1-z)t)}{t}C(t(1-(1-z)t))-\frac{1-2(1-z)t}{t}\notag \\
		&=&(1-2(1-z)t)\sum_{j\geq0}C_jt^{j-1}(1-(1-z)t)^{j+1}-\frac{1-2(1-z)t}{t}\notag \\	&=&(1-2(1-z)t)\sum_{j\geq0}\sum_{i=0}^{j+1}C_j\binom{j+1}{i}t^{i+j+1}(z-1)^i-\frac{1-2(1-z)t}{t}.
		\feqn
		Extracting the coefficient of $t^{n-3}$, we have completed the proof Theorem~\ref{twoside-thm}-(I).
		
		From \eqref{twoside-H1sol}, we have
		$$H_1'(t,1)=\frac{\partial}{\partial z}H_1(t,z)\mid_{z=1}=-2+\frac{1}{t}-\frac{1-5t}{t\sqrt{1-4t}},$$
		which leads to $\frac{1}{C_{n-2}} h'_n=\frac{n(n-1)}{2(2n-5)}$ , for all $n\geq4$. Moreover,
		$$H_1''(t,1)=\frac{\partial^2}{\partial z^2}H_1(t,z)\mid_{z=1}=\frac{2t(2-7t)}{(1-4t)^{3/2}},$$
		which shows
		$\frac{1}{C_{n-2}}h''_n=\frac{n(n-1)(n-2)(n-3)}{4(2n-5)(2n-7)}$ for $n\geq5$. Therefore, $\E(\cals_n)$ and $\Var(\cals_n)$ follow from \eqref{var-e-form}.
		
		\subsection{Degree of vertex $1$}
		
		Recall that conditions for Lemma \ref{lem1} s not satisfied by $f_n$ for this example, however, by a very similar type of argument we can show $h_{n,l,r}(z)=h_{n,2,r-l+2}(z)$ for all $2\leq l<r\leq n-1$ and $n\geq 3$. Then, \eqref{EqT3} implies
		$$h_{n,1,n}(z)=z\sum_{j=2}^{n-1}h_{j,1,j}(z)h_{n-j+2,2,n-j+1}(z),$$
		and
		$$h_{n,2,n}(z)=\sum_{j=3}^{n-1}h_{j,2,j}(z)h_{n-j+2,2,n-j+1}(z),$$
		where $h_{2,1,2}(z)=h_{3,2,3}(z)=h_{4,2,4}(z)=1$. By translating these recurrence in terms of generating functions $H_1(t,z)$ and $H_2(t,z)$, we obtain
		$$H_1(t,z)=zH_2(t,z)+ztH_1(t,z)H_2(t,z)\mbox{ and }H_2(t,z)=1+t(H_2(t,z))^2.$$
		Therefore,
		$$H_1(t,z)=\frac{zC(t)}{1-ztC(t)}\mbox{ and }H_2(t,z)=C(t).$$
		Thus,
		$$H_1(t,z)=\sum_{s\geq1}z^{s}t^{s-1}C^{s}(t),$$
		By Equation 2.5.16 \cite{wilf1}, we obtain		$$t^3H_1(t,z)=\sum_{s\geq0}\sum_{j\geq0}z^s\frac{s(2j+s-1)!}{j!(j+s)!}t^{j+s+2},$$
		which leads to
		$$h_{n}(z)=\sum_{s=0}^{n-2}z^s\frac{s(2n-s-5)!}{(n-s-2)!(n-2)!}$$
		In addition, $\E(\cals_n)$ and $\Var(\cals_n)$ easily follow from \eqref{var-e-form}. This completes the proof of Lemma~\ref{degree-lem}.
		
		Next, we give the proof of Theorem~\ref{degree-thm}.
		\begin{proof}[Proof of Theorem~\ref{degree-thm}]
			Recall $K:=\sum_{i=1}^{n-2} k_i$. Note that
			\beq
			&& \PP(\cala_{n,1} = k_1, \cdots, \cala_{n,n-2} = k_{n-2}) = \PP\left(\cala_{n,1} = k_1, \cdots, \cala_{n,n-2} = k_{n-2}, \cals_{n} = K \right) \\
			&&\quad = \PP\left(\cala_{n,1} = k_1, \cdots, \cala_{n,n-2} = k_{n-2} \ | \  \cals_{n} = K \right) \PP\left(\cals_{n} = K\right)
			\feq
			The last term is given by Lemma~\ref{degree-lem}. Hence, it is enough to calculate the first term of the right-hand side of the equality. Given $k_1, \cdots, k_{n-2}$ and $K$, we count how many triangulations have this portfolio at vertex $1$. Note that there are $\binom{K}{k_1,\cdots, k_{n-2}}$ choices of these angles. For each of these $\binom{K}{k_1,\cdots, k_{n-2}}$ choices, we have $C_{0}^{k_1}C_{1}^{k_2}\cdots C_{n-3}^{k_{n-2}}$ triangulations that fit the description. Therefore,
			\beq
			\PP\left(\cala_{n,1} = k_1, \cdots, \cala_{n,n-2} = k_{n-2} | \cals_{n} = K \right) = \frac{1}{Z_{n,K}}\binom{K}{k_1,\cdots, k_{n-2}} C_{0}^{k_1}C_{1}^{k_2}\cdots C_{n-3}^{k_{n-2}},
			\feq
			where $Z_{n,k}$ is the number of triangulations with $K$ angles at vertex  ``1", defined by \eqref{Znk-def}. This completes the proof.
		\end{proof}
		
		\subsection{Blue angles}
		\begin{proof}[Proof of Lemma~\ref{sumblue-lemma}]				
		Recall \eqref{blue-p-eq}. Lemma \ref{lem1} along with \eqref{EqT3} implies $$h_{n,1,n}(z)=z\sum_{j=2}^{n-1}z^{j-2}h_{j,1,j}(z)h_{n-j+1,1,n-j+1}(z),$$ where $h_{2,1,2}(z)=1$. Then by rewriting this recurrence in terms of the generating function $\tilde{H}_1=\sum_{n\geq2}h_{n,1,n}(z)t^{n-2}$, we have  $$\tilde{H}_1(t,z)=1+zt\tilde{H}_1(zt,z)\tilde{H}_1(t,z).$$ Hence $\tilde{H}_1$ has the following form		$$\tilde{H}_1(t,z)=\dfrac{1}{1-\dfrac{zt}{1-\dfrac{z^2t}{1-\dfrac{z^3t}{\ddots}}}}.$$
		Calculating the derivative at $z=1$, we obtain $\E(\cals_{n})=\frac{2^{2n-5}-\binom{2n-5}{n-2}}{C_{n-2}}$ for $n\geq3$ as claimed in Lemma~\ref{sumblue-lemma}.
		\end{proof}
		
		\begin{proof} [Proof of Theorem~\ref{blue-p-thm}]
		Recall \eqref{sum_blue-eq} and let $p=1$. \eqref{EqT3} along with Lemma \ref{lem1} implies	$$h_{n,1,n}(z)=zh_{n-1,1,n-1}(z)+\sum_{j=3}^{n-1}h_{j,1,j}(z)h_{n-j+1,1,n-j+1}(z),$$
		where $h_{2,1,2}(z)=1$ and $h_{3,1,3}(z)=z$. By multiplying  by $t^{n-3}$ and summing over $n\geq4$, we obtain
		$$tH_1(t,z)=z(tH_1(t,z)+1)+t(tH_1(t,z)+1)H_1(t,z).$$
		By solving this equation, we obtain
		$$H_1(t,z)=\frac{1-t-zt-\sqrt{t^2(1-z)^2-2t(1+z)+1}}{2t^2}.$$
		Once again extracting the coefficient of $t^{n-3}$ in $H_1(t,z)$ completes the proof of Theorem~\ref{blue-p-thm}.
		\end{proof}

			\end{document}